\documentclass[12pt]{amsart}
\usepackage[latin1]{inputenc}
\usepackage{indentfirst}
\usepackage{amsfonts}
\usepackage{graphicx}
\usepackage{amsmath}
\usepackage{amsfonts}
\usepackage{amssymb}
\usepackage{latexsym}
\usepackage{amsthm}
\usepackage{amscd}
\usepackage{color,soul}
\usepackage[all]{xy}
\usepackage[normalem]{ulem}
\newcommand{\cL}{\mathcal L}
\newcommand{\cD}{\mathcal D}

\newcommand{\fq}{\mathbb {F}_q}
\newcommand{\Z}{\mathbb {Z}}
\newcommand{\bP}{\mathbb {P}}
\newcommand{\N}{\mathbb {N}}

\newtheorem{theorem}{Theorem}[section]

\newtheorem{lemma}[theorem]{Lemma}

\newtheorem{proposition}[theorem]{Proposition}

{\theoremstyle{definition}
\newtheorem{example}[theorem]{Example}}

\newcommand{\rmv}[1]{}

\title{One- and Two-Point Codes over Kummer Extensions}

\begin{document}

\author{Ariane M. Masuda, Luciane Quoos, \and Alonso Sep\'ulveda}
\address[A.~M.~Masuda]{Department of Mathematics \\ New York City College of Technology, CUNY \\ 300 Jay Street\\ Brooklyn, NY 11201\\ USA}
\email{amasuda@citytech.cuny.edu}
\address[L. Quoos]{Instituto de Matem\'atica \\ Universidade Federal do Rio de Janeiro \\ Cidade Universit\'aria \\ Rio de Janeiro, RJ 21941-909\\ Brazil }
\email{luciane@im.ufrj.br}
\address[A. Sep\'ulveda]{Faculdade de Matem\'atica\\ Universidade Federal de Uberl\^andia \\ Avenida Jo\~ao Naves de Avila, 2121 \\Uberl\^andia, MG  38408-100 \\Brazil}
{\email{alonso.castellanos@ufu.br}
\thanks{}
\date{\today}
\thanks{The first author was partially supported  by  PSC-CUNY Award \#68121-00 46, jointly funded by The Professional Staff Congress and The City University of New York. The second author was partially supported by CNPq, PDE-200434/2015-2. The third author would like to thank FAPEMIG, APQ-00506-14.}

\begin{abstract}

We compute the Weierstrass semigroup at one totally ramified place for Kummer extensions defined by $y^m=f(x)^{\lambda}$ where $f(x)$ is a separable polynomial over $\fq$. In addition, we compute the Weierstrass semigroup at two certain totally ramified places.  We then apply our results to construct one- and two-point algebraic geometric codes with good parameters.
\end{abstract}

\maketitle

\section{Introduction}
Using tools from algebraic geometry, Goppa~\cite{goppa1} constructed error-correcting linear codes that are nowadays known as algebraic geometric (or AG) codes. Let $\mathcal X$ be a non-singular, projective, geometrically irreducible, algebraic curve of genus $g$ over a finite field $\fq$.  Let $D$ and $G$ be divisors on $\mathcal{X}$ such that $D$ is the sum of $n$ distinct rational points on $\mathcal X$  that are not in the support of $G$; say, $P_1, \ldots, P_n$. There are two classical ways of constructing AG codes using $D$ and $G$. One depends on the Riemann-Roch space $\mathcal{L}(G)$,
  $$
C_{\mathcal{L}}: =C_{\mathcal{L}}(\mathcal{X},D,G)=\{ (f(P_1),\ldots,f(P_n))\mid f\in \mathcal{L}(G)\} \subseteq \fq^n, $$
 and the other one is based on  the space of differentials $\Omega(G-D)$  on $\mathcal{X}$,
 $$C_{\Omega}:= C_{\Omega}(\mathcal{X},D,G)=\{(\mbox{res}_{P_1}(\eta),\ldots,\mbox{res}_{P_n}(\eta))\mid \eta\in \Omega(G-D)\}\subseteq \fq^n.
  $$

\noindent It turns out that $C_{\mathcal{L}}$ and $C_{\Omega}$ are dual codes.

We denote by $d_{\mathcal{L}}$ and  $d_\Omega$ the minimum distance of $C_{\mathcal{L}}$ and $C_{\Omega}$, respectively.  It is well known that
$d_{\mathcal{L}} \geq n- \deg G$ and $d_\Omega \geq \deg G - (2g-2)$
where $\deg G$ is the degree of $G$.
These bounds have been improved when $G$ is a one-point divisor supported on a rational point $P$; see~\cite{goppa2}. In~\cite{GKL,GKL1} Garcia, Kim and Lax observed that these improvements were related to the arithmetical structure of the Weierstrass semigroup at $P$. More specifically, if $\beta, \beta+1, \ldots, \beta+t$ and $\gamma -(t-1),\ldots, \gamma-1, \gamma$ are sequences of consecutive gaps at $P$ with $t \geq 1$ and $\beta+t \leq \gamma$, then for $G=(\beta+\gamma-1)P$ we have
that $d_\Omega \geq  \deg G -(2g - 2) + t + 1.$
A similar result for two-point codes was obtained by Homma and Kim using the notion of pure gaps that they introduced in~\cite{HK}; see Theorem~\ref{puregapscodes}.

Many authors have investigated the minimum distances of one- and two-point codes over specific Kummer extensions with the Hermitian function field being the most well-known one. For instance, Maharaj~\cite{M} studied one-point codes over subcovers of the Hermitian curve. Homma, Kim and Matthews~\cite{HK,GM2} studied two-point codes and applied their results to the Hermitian curve. In~\cite{MST} Munuera, Sep\'ulveda and Torres considered codes over a Castle curve, which is a generalization of the Hermitian curve.
Recently,  Sep\'ulveda and Tizziotti~\cite{ST} studied two-point codes over the $\mathbb{F}_{q^{2\ell}}$-maximal curve whose affine plane model is  given by $y^{q^\ell+1}=x^q+x$.

In this paper we study the Weierstrass semigroup in one and two points over Kummer extensions given by $y^m=f(x)^{\lambda}$ where
 $f(x)$ is a separable polynomial over $\fq$ of degree $r$  and  $\gcd(m, r\lambda)=1$. Then we apply these results to construct one- and two-point codes over Kummer extensions. Note that we do not specify the polynomial $f(x)$ as many authors do, and so our results are very general. In particular, for implementation purposes, this gives a great deal of flexibility as the algebraic curve can be carefully chosen in order to improve either the parameters of the code or the code performance.

The organization of the paper is as follows. In Section 2 we present some preliminary results on Kummer extensions. In Section 3 we obtain the Weierstrass semigroup at a totally ramified place using Proposition~\ref{main} as the main ingredient. In Section 4 we compute the Weierstrass semigroup at two points, $P$ and $P_{\infty}$, where $P$ is a totally ramified point in a Kummer extension. We also give conditions to obtain pure gaps. Finally, in Section 5 we use our results to construct one- and two-point codes that have better parameters when compared with the corresponding ones in the MinT's Tables~\cite{MinT}.

\section{Preliminary results}

Let $\fq$ be a finite field of characteristic $p$, $K$ be the algebraic closure of $\fq$, and $F/K$ be an algebraic function field in one variable of genus $g$. We denote by $\bP_F$ the set of places of $F$ and by $\cD_F$ the free abelian group generated by the places of $F$. An element in $\cD_F$ is called a divisor. By writing a divisor as
$
D=\sum\limits_{P\in \bP_F} n_P\,P$  with $n_P\in \Z $ and almost all  $n_P=0$, we have that
 the degree of $D$ is $\deg D=\sum\limits_{P\in \bP_F}n_P.$
The Riemann-Roch vector space over $\fq$ associated to $D$ is defined by
$
\cL(D)=\{z\in F\mid (z)\ge -D\}\cup \{0\}
$
where $(z)$ is the divisor of the function $z$.  The  dimension of $\cL(D)$ is denoted by $\ell(D)$. From the Riemann-Roch Theorem, it follows that
$\ell(D) = \deg D+1-g$ if  $\deg D\ge 2g-1$.
For $P$ in $\bP_F$, the Weierstrass  semigroup at $P$ is
$H(P)=\{n\in \N_0\mid  (z)_\infty=nP \text{ for some } z\in F\}$
 where $(z)_\infty$ is the pole divisor of the function $z$.
A non-negative integer $n$ is called a non-gap at $P$ if $n\in H(P)$, and a gap at $P$ otherwise.
The Weierstrass Gap Theorem asserts that $s$ is a gap at a place $P$ of degree one  if and only if $\ell((s-1)P)=\ell(sP)$. As a consequence, there exist $g$ gaps at $P$, say $s_1,\ldots,s_g$, satisfying $1=s_1<\cdots<s_g\le 2g-1$. We let $G(P)$ be the set of such $g$ gaps.

Next we extend the notion to that of a non-gap at two distinct rational places $P_{1}$ and $P_{2}$  in $\bP_F$. The Weierstrass semigroup at $(P_{1}, P_{2})$ is the set
$$
H(P_{1}, P_{2}) = \{(n_{1}, n_{2}) \in \mathbb{N}_0^2 \mid (z)_{\infty} = n_{1} P_{1} + n_{2} P_{2} \text{ for some } z\in F\}.
$$
Analogously, the set $G(P_{1}, P_{2})= \mathbb{N}_0^2 \setminus H(P_{1}, P_{2})$ is called the  Weierstrass gap set at $(P_{1}, P_{2})$.

Unlike the one-point case, the cardinality of $G(P_{1}, P_{2})$ depends on the choice of the points $P_{1}$ and $P_{2}$; see~\cite{kim}. Homma~\cite{H} found bounds for the cardinality of $G(P_{1}, P_{2})$ as well as a connection between $H(P_{1}, P_{2})$ and a permutation of the set $\{1,2, \ldots, g\}$ that we will discuss in Section~\ref{semigroup}.

An important subset of $G(P_{1}, P_{2})$ is the set of pure gaps.  A pair of non-negative integers $(s_1,s_2)$ is said to be a pure gap at $(P_1, P_2)$ if
$$\ell (s_1P_1+s_2P_2)= \ell((s_1-1)P_1+s_2P_2)=\ell(s_1P_1+(s_2-1)P_2).$$
From~\cite[Proposition 2.4]{HK}, we have that $(s_1,s_2)$ is a pure gap at $(P_1, P_2)$ if and only if
$$\ell (s_1P_1+s_2P_2)= \ell((s_1-1)P_1+(s_2-1)P_2).$$ We denote by $G_0(P_1,P_2)$ the set of pure gaps at $(P_1,P_2)$.
Pure gaps can be used to improve the bound on the  minimum distance of AG codes.

\begin{theorem}[{\cite[Theorem 3.3]{HK}}]\label{puregapscodes}
Let $P_1$ and $P_2$ be rational points. For $t_1,t_2\in\N$, suppose that
$$\{(k_1, k_2) \mid \beta \leq k_1 \leq \beta+t_1\text{ and } \gamma \leq k_2\leq \gamma + t_2\} \subseteq G_0(P_1, P_2).$$
 If $G=(2\beta+t_1-1)P_1+(2\gamma + t_2-1)P_2$, then
$d_\Omega \geq \deg G -(2g - 2) + t_1 +t_2 + 2.$
\end{theorem}

A Kummer extension is a field extension $F/K$ defined by
$y^m=f(x)^{\lambda}=\prod\limits_{i=1}^{r}(x-\alpha_i)^{\lambda}$
 such that $m \geq 2$, $p\nmid m$,  $f(x) \in \fq[x]$, $\alpha_1,\ldots,\alpha_r \in K$ are pairwise distinct and $\gcd(m,r\lambda)=1$.
Let $Q_1, \dots , Q_r$ denote the places of the rational function field $K(x)$ associated to the zeros $\alpha_1,\ldots,\alpha_r$ of $f(x)$,  respectively, and let $Q_\infty \in \mathbb{P}_{K(x)}$ be the only pole of $x$. Since they are totally ramified in the extension $F/K(x)$, we denote by $P_i$ the unique place in $\mathbb{P}_F$ over $Q_i$ for $i=1, \dots ,r, \infty$. We note that the curve has genus $g=(m-1)(r-1)/2$; see~\cite[Proposition 3.7.3]{Sti}.

We build on recent work of Abd\'on, Borges and Quoos where they obtained an arithmetical criterion for an integer to be a gap number at a totally ramified point $P$.  Below we state their result restricted to the setting in which we are interested here.
\begin{proposition}[{\cite[Corollary 3.6]{ABQ}}]\label{main}
Let $y^m=f(x)^{\lambda}$ be a Kummer extension where
 $f(x)$ is a separable polynomial over $\fq$ of degree $r$ and  $\gcd(m, r\lambda)=1$.
A non-negative integer $s$ is a gap number  at a totally ramified place $P$ if and only if
$$r\left\lbrace\frac{t\lambda}{m}\right\rbrace>1+\left\lfloor\frac{s-1}{m}\right\rfloor$$
where $t \in \{0,\dots ,m-1 \}$ is the only solution to  $s+\lambda t \equiv 0 \pmod{m}$.
\end{proposition}

We will also use a result due to Maharaj~\cite{M} to compute $G_0(P_1,P_2)$ at two rational points of $F$.  In order to do that, we will need
another definition. For any field $E$ with $K \subseteq E\subseteq F$, write the divisor $D$ of $F$ as
$$
D = \sum\limits_{R\in \bP_E}\,\sum_{\substack{{Q\in \bP_F}\\{Q|R}}}\, n_Q\, Q.
$$
We define the restriction of $D$ to $E$ as
$$
D\Big|_{E}= \sum\limits_{R\in \bP_E} \min\,\left\{\left\lfloor\frac{n_Q}{e(Q|R)}\right\rfloor\colon
{Q|R}\right\}\,R
$$
where $e(Q|R)$ is the ramification index of $Q$ over $R$.

\begin{theorem}[{\cite[Theorem 2.2]{M}}] \label{ThMaharaj}
Let $F/K$ be a Kummer extension of degree $m$ defined by $y^m=f(x)$.  Then, for any divisor $D$ of $F$ that is invariant by the action of $Gal(F/K(x))$, we have that
$$ \mathcal{L}(D)= \bigoplus\limits_{t=0}^{m-1} \mathcal L\left([D+(y^t)]\Big|_{K(x)}\right)\,y^t$$
where $[D+(y^t)]\Big|_{K(x)}$ denotes the restriction of the divisor $D+(y^t)$ to $K(x)$.
\end{theorem}


\section{The Weierstrass Semigroup $H(P)$}

In this section we calculate the Weierstrass semigroup $H(P)$ for  a rational place $P\in\mathbb{P}_F$.

\begin{proposition}\label{divisor} Let $F/K$ be a  Kummer extension given by $y^m=f(x)^{\lambda}=\prod\limits_{i=1}^{r}(x-\alpha_i)^{\lambda}$ where $f(x)$ is a separable polynomial of degree $r$ and $\gcd(m,r\lambda)=1$. Then we have the following divisors in $F$:
\begin{enumerate}
\item $(x-\alpha_i)=mP_i-mP_\infty$  for every $i$, $1\le i \le r$,
\item $(y)=\lambda P_1+\cdots + \lambda P_r-r\lambda P_\infty$,
\item $(f(x))=\sum\limits_{i=1}^rmP_i - rmP_\infty$, and
\item $(z)_\infty=rP_\infty$ for some function $z \in F$.
\end{enumerate}
\end{proposition}

\begin{proof}
Items (1) and (2) follow from the definition of the curve, and (3) follows from (1). Since $\gcd(m,\lambda)=1$ and we can suppose that
$0<\lambda<m$, there exist $a,b \geq 1$ such that $am-b\lambda=1$. Thus   \[(f(x)/y)=(m-\lambda)
P_1 +\cdots +(m- \lambda)   P_r -r(m -\lambda) P_\infty,\]
and it follows that $\left(\dfrac{f(x)^{a}}{y^{b}}\right)_{\infty}=r(am-b\lambda)P_\infty=rP_\infty$. This shows (4).
\end{proof}

Abd\'on, Borges and Quoos determined the Weierstrass semigroup at totally ramified places for Kummer extensions
defined  by $y^m=f(x)$ where $f(x)$ is a separable polynomial~\cite[Theorem 6.1]{ABQ}. Next we extend their result.

\begin{theorem}\label{pesos}
Let $F/K$ be a  Kummer extension given by
$y^m=f(x)^{\lambda}$ where  $f(x)$ is a separable polynomial of degree $r$ with $\gcd(m,r\lambda)=1.$
Suppose that the genus $g$ satisfies $g \geq 1$ and $P$ is a totally ramified place  in $\mathbb{P}_F$. Then:
\[H(P)=
\begin{cases}
\mathbb{N}_0 \setminus \left\{1+i+mj \,|\, 0 \leq i \leq m-2-\lfloor m/r\rfloor, \, 0\leq j \leq r-2- \lfloor r(i+1)/m\rfloor\right \}, &\hspace{-0.5cm}\text{ if } P\ne P_\infty \\
\langle m,r\rangle, &\hspace{-0.5cm}\text{ if } P=P_\infty.
\end{cases}\]
\end{theorem}

\begin{proof}
We begin with the case $P\neq P_\infty$.
Let $s \in \mathbb{N}_0$ and $t \in \{0,\ldots,m-1\}$ be such that $t\lambda = m-s \pmod m$.
By writing $s-1=m\lfloor(s-1)/m\rfloor+i$ with $0 \leq i \leq m-1$, we obtain that
$m-s=m\left(-\lfloor(s-1)/m\rfloor\right)+m-i-1$, and since $0 \leq m-i-1 \leq m-1$, we conclude that $\lbrace t\lambda/m\rbrace=(m-i-1)/m$.
From Proposition~\ref{main}, we have that $s=m\lfloor(s-1)/m\rfloor+i+1$ is a gap at $P$ if and only if
$$\left\lfloor\frac{s-1}{m}\right\rfloor+ 1<\frac{r(m-i-1)}{m},$$
that is, $\left\lfloor(s-1)/m\right\rfloor<r-1-r(i+1)/m$. Therefore, the gap numbers at $P$ are of the form $s=mj+i+1$ where $0\le i \le m-1$ and $0\leq j < r-1-r(i+1)/m$. In particular,  $r-1-r(i+1)/m>0$ implies that $i<m-1-m/r$. In addition, since $i$ and $j$ are integers, we have the following intervals: $0\le i \le m-2-\lfloor m/r \rfloor$ and $0\le j \le r-2- \lfloor r(i+1)/m\rfloor$.

Finally, if $P$ is over the place $ Q_\infty \in \mathbb{P}_{K(x)}$, then by Proposition~\ref{divisor} we have that the semigroup $\langle m,r\rangle$ is contained in  $H(P_\infty)$. Since $\langle m,r\rangle$ has genus $(m-1)(r-1)/2$, which is also  the genus of the curve, we conclude that $H(P_\infty)=\langle m,r\rangle$.
\end{proof}

We observe that for a totally ramified place $P$ except for $P_\infty$, by setting $j=0$ and $m-1-\lfloor m/r\rfloor \leq i \leq m-1$, the consecutive numbers
$m-\lfloor m/r\rfloor, m-\lfloor m/r\rfloor+1, \dots, m$ are non-gaps at $P$. Moreover, in some instances, this sequence generates the entire Weierstrass non-gap semigroup. That is the case in~\cite[Proposition 1]{ST} where Sep\'ulveda and Tizziotti  showed that, for $P=(0:0:1)$ on the curve $y^{q^\ell+1}=x^q+x$ over $\mathbb{F}_{q^{2\ell}}$ and $\ell$ odd,  one has
$$H(P)=\langle q^\ell-q^{\ell-1}+1,q^\ell-q^{\ell-1}+2,\ldots,q^\ell,q^\ell+1\rangle.$$
As a consequence of Theorem~\ref{pesos}, we are able to generalize this result. We will use a result that allows us to compute the genus of a semigroup generated by consecutive integers.

\begin{proposition}[{\cite[Remark 2.6]{CT}}]
\label{semigrupo}
A semigroup $H$  generated by $t$ consecutive numbers, $n, n+1, \dots , n+t-1$, has genus $g(H)= \frac{J(2(n-1)-(J-1)(t-1))}{2}$ where $J= \big\lceil \frac{n-1}{t-1}\big\rceil$.
\end{proposition}

\begin{theorem}\label{casobom}
Let $F/K$ be a  Kummer extension $y^m=f(x)^{\lambda}$ where $f(x)$ is a separable polynomial of degree $r$ and $m=rt +1$ with $t \geq 1$. If $P\neq P_\infty $ is a totally ramified place then
\begin{equation}\label{hp}
H(P)=\langle m-\lfloor m/r\rfloor, m-\lfloor m/r\rfloor+1, \dots, m\rangle.
\end{equation}
\end{theorem}
\begin{proof}
Let $H=\langle m-\lfloor m/r\rfloor, m-\lfloor m/r\rfloor+1, \dots, m\rangle$. By choosing $j=0$ and $ m-\lfloor m/r\rfloor \leq i+1 \leq m$ in Theorem~\ref{pesos},
one sees that $H \subseteq H(P)$. The genus  of the curve $y^{m}=f(x)^{\lambda}$ is $g=(m-1)(r-1)/2$.  On the other hand, from Proposition~\ref{semigrupo}, the genus of the semigroup $H$ and the genus of the curve are the same, and so $H=H(P)$.
\end{proof}

Let $t$ and $r$ be integers such that $t\ge 1$ and $r\ge 2$.  For $m=rt +1$,  the Weierstrass semigroup~(\ref{hp}) is symmetric exactly when $m=r+1$. In fact, by Theorem~\ref{pesos},  the largest gap occurs when $j=r-2$ and $i=t-1$. One can check that this gap is equal to $2g-1$ if and only if $t=1$. A computer experiment suggests that the only values of $m$ for which (\ref{hp}) holds are the ones described in Theorem~\ref{casobom}.

\section{The Weierstrass Semigroup $H(P_1,P_2)$}\label{semigroup}

Let $P_{1}$ and $P_{2}$ be rational points in $\mathbb{P}_F$, and $\beta_{1} < \beta_{2} < \cdots < \beta_{g}$  and
 $\gamma_{1} < \gamma_{2} < \cdots < \gamma_{g}$ be the gap sequences at $P_{1}$ and $P_{2}$, respectively. For each $i$, we let $n_{\beta_i} = \min \{ \gamma \in \mathbb{N}_{0} \mid (\beta_i, \gamma ) \in H(P_{1}, P_{2}) \}$. Then $\{ n_{\beta} \mid \beta \in G(P_{1}) \} = G(P_{2})$ by~\cite[Lemma 2.6]{kim}.  So there exists a permutation $\sigma$ of the set $\{1,2, \ldots , g\}$ such that $n_{\beta_{i}} = \gamma_{\sigma(i)}$. The graph of the bijective map between $G(P_{1})$ and $G(P_{2})$ is the set
$
\Gamma(P_{1}, P_{2})  = \{ (\beta_{i} , n_{\beta_i}) \mid  i=1,2, \ldots ,g\} = \{ (\beta_{i} , \gamma_{\sigma(i)}) \mid  i=1,2, \ldots ,g \}.
$

\begin{lemma}[{\cite[Lemma 2]{H}}] \label{lemma 1}
Let $\Gamma '$ be a subset of $(G(P_{1}) \times G(P_{2})) \cap H(P_{1},P_{2})$. If there exists a permutation $\tau$ of $\{ 1, 2, \ldots , g\}$ such that $\Gamma ' = \{ (\alpha_{i} , \beta_{\tau(i)}) \mid  i=1,2, \ldots ,g \}$, then $\Gamma ' = \Gamma(P_{1}, P_{2})$.
\end{lemma}

Given $\Gamma (P_{1}, P_{2})$, we can compute $H(P_{1}, P_{2})$ in the following way. For $\mathbf{x} = (\beta_{1}, \gamma_{1})$ and $\mathbf{y} = (\beta_{2}, \gamma_{2}) \in \mathbb{N}_{0}^2$, the \textit{least upper bound} of $\mathbf{x}$ and $\mathbf{y}$ is defined as $\mathrm{lub}(\mathbf{x},\mathbf{y})= (\max\{\beta_{1}, \beta_{2}\}, \max\{\gamma_{1}, \gamma_{2}\})$.
In~\cite{kim} it was shown that if $\mathbf{x},\mathbf{y} \in H(P_{1}, P_{2})$ then $\mathrm{lub}(\mathbf{x},\mathbf{y}) \in H(P_{1}, P_{2})$. Moreover, we have the following result.

\begin{lemma}[{\cite[Lemma 2.2]{kim}}] \label{lemma 2}
Let $P_{1}$ and $P_{2}$ be two distinct rational points. Then
$$H(P_{1},P_{2}) = \{ \mathrm{lub} (\mathbf{x},\mathbf{y}) \mid \mathbf{x},\mathbf{y} \in \Gamma(P_{1}, P_{2}) \cup (H(P_{1}) \times \{0\}) \cup (\{0\} \times H(P_{2})) \}.$$
\end{lemma}

Next we describe the graph of the bijective map between $G(P_{\infty})$ and $G(P)$ where $P$ is a totally ramified place.

\begin{theorem} \label{beta_ij}
Let $F/K$ be a  Kummer extension defined by $y^m=f(x)^{\lambda}$ where $f(x)$ is a separable polynomial of degree 
$r$ with $\gcd(m,r\lambda)=1$.
For a totally ramified place $P$,  we have that
$$\Gamma(P_{\infty}, P) = \left\{ (mr-mj-ri, i+m(j-1)) \mid  1 \leq i \leq m-1-\lfloor m/r\rfloor, \, 1\leq j \leq r-1- \lfloor ri/m\rfloor \right\}.$$
\end{theorem}

\begin{proof}
Let $\Gamma'=\{ (mr-mj-ri, i+m(j-1)) \mid  1 \leq i \leq m-1-\lfloor m/r\rfloor, \, 1\leq j \leq r-1- \lfloor ri/m\rfloor \}$.  By  Theorem~\ref{pesos}, we have that $\Gamma'\subseteq G(P_{\infty})
 \times G(P)$ and the cardinality of $\Gamma'$ is $g$.

Without loss of generality, we can assume that $P=P_{1}$. Since $\gcd(m, \lambda)=1$, there exist integers $a$ and $b$ such that $a\lambda +bm=1$. The divisor of the function $z=y^{a(m-i)}f(x)^{b(m-i)}(x-\alpha_1)^{-j}$ is 
$$ (z)=\sum_{s=2}^r (m-i)P_s-(rm-ri-jm)P_\infty -(m(j-1)+i)P_1,$$
and so we conclude that $\Gamma' \subseteq H(P_{\infty}, P_{1})$. By Lemma~\ref{lemma 1}, $\Gamma'= \Gamma(P_{\infty}, P)$.
\end{proof}

The following result illustrates how we can identify pure gaps at $(P_\infty, P)$.

\begin{theorem}\label{puregaps}
Let  $P$ be a totally ramified place in the Kummer extension $F/K$ and $(a,b) \in \mathbb N^2$. Then $(a,b)$ is a pure gap at $(P_\infty, P)$ if and only if, for every  $t\in\{0, \dots , m-1\}$, exactly one of the following two conditions is satisfied:
\begin{enumerate}
\item $\left\lfloor \frac{a-rt}{m} \right\rfloor + \left\lfloor \frac{b+t}{m} \right\rfloor < 0$
\item $\left\lfloor \frac{a-rt}{m} \right\rfloor  + \left\lfloor \frac{b+t}{m} \right\rfloor \geq 0$ and
$\left\lfloor \frac{a-rt}{m}\right \rfloor  + \left\lfloor \frac{b+t}{m} \right\rfloor = \left\lfloor \frac{a-1-rt}{m} \right\rfloor + \left\lfloor \frac{b-1+t}{m} \right\rfloor$.
\end{enumerate}
\end{theorem}
\begin{proof}
For convenience, we suppose $P=P_1$. From Theorem~\ref{ThMaharaj}, we have
$$\cL(aP_\infty+bP_1)= \bigoplus_{t=0}^{m-1} \cL\left([aP_\infty+bP_1+ (y^t)]\Big |_{K(x)}\right) y^t.$$
Since $(y)=\displaystyle\sum\limits_{i=1}^r P_i - rP_\infty$, it follows that
$$aP_\infty+bP_1+ (y^t)=(a-rt)P_\infty + (b+t)P_1 + \sum\limits_{i=2}^r tP_i.$$
The restriction of this divisor to the rational function field $K(x)$ is
$$[aP_\infty+bP_1+ (y^t)]\Big |_{K(x)}= \left\lfloor \frac{a-rt}{m}\right\rfloor Q_\infty  + \left\lfloor \frac{b+t}{m} \right\rfloor Q_1,$$
given that $\lfloor t/m \rfloor =0$ for all $t=0, \dots, m-1$.
Hence
$$\ell(aP_\infty+bP_1)= \displaystyle\sum\limits_{t=0}^{m-1} \ell\left(\left\lfloor \frac{a-rt}{m}\right \rfloor Q_\infty  + \left\lfloor \frac{b+t}{m} \right\rfloor Q_1\right).$$
In a similar way, we have that
$$\ell((a-1)P_\infty+(b-1)P_1)= \displaystyle\sum\limits_{t=0}^{m-1} \ell \left(\left\lfloor \frac{a-1-rt}{m}  \right\rfloor Q_\infty + \left\lfloor \frac{b-1+t}{m} \right\rfloor Q_1\right).$$
Now $\ell (aP_\infty+bP_1)= \ell((a-1)P_\infty+(b-1)P_1)$ if and only if $\sum\limits_{t=0}^{m-1} \ell_t=0$ where
$$\ell_t= \ell\left(\left\lfloor \frac{a-rt}{m} \right\rfloor Q_\infty  + \left\lfloor \frac{b+t}{m} \right\rfloor Q_1\right)-\ell\left(\left\lfloor \frac{a-1-rt}{m}\right\rfloor Q_\infty  + \left\lfloor \frac{b-1+t}{m} \right\rfloor Q_1\right) \in \{0,1,2\}.$$
Using the Riemann Theorem and the fact that $K(x)$ has genus zero, $\ell_t=0$ if and only if  either
$$\left\lfloor \frac{a-rt}{m} \right\rfloor + \left\lfloor \frac{b+t}{m} \right\rfloor < 0$$
or$$\left(\left\lfloor \frac{a-rt}{m} \right\rfloor  + \left\lfloor \frac{b+t}{m} \right\rfloor \geq 0  \quad \text{and} \quad
\left\lfloor \frac{a-rt}{m} \right\rfloor  + \left\lfloor \frac{b+t}{m} \right\rfloor = \left\lfloor \frac{a-1-rt}{m} \right\rfloor + \left\lfloor \frac{b-1+t}{m} \right\rfloor\right).$$
\end{proof}

As a consequence of Theorem~\ref{puregaps}, we exhibit a family of  pure gaps.

\begin{proposition}\label{puregap1}
Let $F/K$ be a  Kummer extension defined by $y^{q^\ell+1}=f(x)$ where $f(x)$ is a separable polynomial of degree $q$ and $q>3$. Then $(q^{\ell+1}-2q^{\ell}-2,1)$ is a pure gap at $(P_{\infty},P)$.
\end{proposition}
\begin{proof}
Let $a:=q^{\ell+1}-2q^{\ell}-2$.  We start by evaluating $\left\lfloor(a-qt)/(q^\ell+1)\right\rfloor$ for
$t\in\{0,\ldots,q^{\ell}\}$. We partition the set $\{0,\ldots,q^\ell-1\}$ as
$$\bigcup\limits_{i=0}^{q-1} \{iq^{\ell-1},\ldots,(i+1)q^{\ell-1}-1\}.$$
If $t\in\{0,\ldots,q^\ell-1\}$ then
\[a-qt =(q-3-i)(q^\ell+1)+(i+1)(q^{\ell}+1)-qt-q.\]
Since $iq^{\ell-1}\le t \le (i+1)q^{\ell-1}-1$ for some $0\le i \le q-1$, it follows that $-(i+1)q^\ell+q \le -qt \le -iq^\ell$ and so
$0 < (i+1)(q^\ell+1)-qt-q <q^\ell+1$.  For $t=q^\ell$, we obtain that $a-qt = -2(q^\ell+1)$. Therefore:
\begin{equation}\label{A}
\left\lfloor\dfrac{a-qt}{q^\ell+1}\right\rfloor=
\begin{cases}
q-3-i, &\text{if } t = iq^{\ell-1},\ldots,(i+1)q^{\ell-1}-1 \text{ with } i = 0,\ldots,q-1 \\
-2, &\text{if } t=q^\ell.
\end{cases}
\end{equation}
On the other hand,
\begin{equation}\label{B}
\left\lfloor\dfrac{1+t}{q^\ell+1}\right\rfloor=
\begin{cases}
0, &\text{if } t = 0,\ldots, q^\ell-1\\
1, &\text{if } t = q^\ell.
\end{cases}
\end{equation}
By combining (\ref{A}) and (\ref{B}), we get
\begin{equation*}\label{AplusB}
\left\lfloor\dfrac{a-qt}{q^\ell+1}\right\rfloor+\left\lfloor\dfrac{1+t}{q^\ell+1}\right\rfloor=
\begin{cases}
q-3-i, &\text{if } t = iq^{\ell-1},\ldots,(i+1)q^{\ell-1}-1 \text{ with } i = 0,\ldots,q-1 \\
-1, &\text{if } t=q^\ell.
\end{cases}
\end{equation*}
This sum is greater than or equal to zero if and only if $t\in\bigcup\limits_{i=0}^{q-3} \{iq^{\ell-1},\ldots,(i+1)q^{\ell-1}-1\}$. In this case,
by using a similar argument, we can show that
\[\left\lfloor\dfrac{a-1-qt}{q^\ell+1}\right\rfloor=q-3-i\qquad\text{and}\qquad  \left\lfloor\dfrac{t}{q^\ell+1}\right\rfloor=0.\]
By Theorem~\ref{puregaps}, we conclude that $(q^{\ell+1}-2q^{\ell}-2,1)$ is a pure gap at $(P_{\infty},P)$.
\end{proof}


\section{Examples of Codes in Kummer Extensions}

In this section we show that our results can yield examples of AG codes that have the same or better parameters than the corresponding ones in the MinT's Tables~\cite{MinT}. 

\begin{example}
The curve $y^{q^\ell+1}=x^q+x$ over $\mathbb{F}_{q^{2\ell}}$ with $\ell$ odd has $q^{2\ell +1}+1$ rational points and genus $g=q^\ell(q-1)/2$. Let $P_{a,b}$ denote the common zero of $x-a$ and $y-b$ where $a,b \in \mathbb{F}_{q^{2\ell}}$.  Let
$b_{1}, \ldots, b_{q^\ell+1}$ be the solutions in $\mathbb{F}_{q^{2\ell}}$ to $y^{q^\ell+1} = a^q + a \neq 0$.
The divisor of $x-a$ is
$$(x-a)=
\begin{cases}
P_{a,b_1}+ \cdots + P_{a,b_{q^\ell + 1}} - (q^\ell + 1) P_{\infty},  &\text{ if } a^q + a \neq 0 \\
(q^\ell + 1)P_{a,0} - (q^\ell + 1) P_{\infty}, &\text{ otherwise}.
\end{cases}
$$
The divisor of $y-b$ is
$
(y-b)= P_{a_1,b}+ \cdots + P_{a_q,b} - q P_{\infty}
$
where $a_{1}, \ldots, a_{q}$ are the solutions in $\mathbb{F}_{q^{2\ell}}$ to $x^q + x = b^{q^\ell + 1}$. 
Consider the divisor $D$ as the sum of all rational points with exception of $P_\infty$ and $P_{0,0}$. It is clear that $n=\deg(D)=q^{2\ell+1}-1$. By Proposition~\ref{puregap1}, we have that  $((q-2)(q^\ell+1)-q, 1)$ is a pure gap at $(P_\infty,P_{0,0})$.
Take $G=(2(q^{\ell +1}-2q^\ell-2)-1)P_\infty + P_{0,0}$. By Theorem 2.1, the minimum distance of  $C_{\Omega}$ satisfies $d_\Omega \geq q^\ell(q-3)$. For $q>3$, we have that $2g-2=q^{\ell+1}-q^\ell-2 < \deg G<n=q^{2\ell+1}-1$, and the dimension of the code is $k=n+g-1-\deg G=(2q^{2\ell+1}-3q^{\ell+1}+7q^\ell+4)/2$.
Hence, the code $C_{\Omega}$ has parameters $[q^{2\ell+1}-1, (2q^{2\ell+1}-3q^{\ell+1}+7q^\ell+4)/2, \geq q^\ell(q-3)]$ over $\mathbb{F}_{q^{2\ell}}$.
Consider the differential
\[\eta=\dfrac{y}{\prod\limits_{i=1}^{q-3}(x-a_i)}dy\]
where  $a_i^q+a_i= 0$ with $a_i\in \mathbb{F}_{q^{2\ell}}\setminus\{0\}$.
By~\cite[Remark 4.3.9]{Sti}, we have $(dy)=(2g-2)P_\infty=(q^{\ell+1}-q^\ell-2)P_\infty$ and
\begin{eqnarray*}
(\eta) &=& P_{0,0}+P_{a_1,0}+ \cdots + P_{a_{q-1},0} - q P_{\infty} - \left(\sum\limits_{i=1}^{q-3}\left[(q^\ell+1)P_{a_i,0} - (q^\ell + 1) P_{\infty}\right]\right)\\
&& +(q^{\ell+1}-q^\ell-2)P_\infty\\
&=&(2q^{\ell+1}-4q^\ell-5)P_\infty +P_{0,0}+ P_{a_{q-2},0} + P_{a_{q-1},0} - q^\ell\sum\limits_{i=1}^{q-3} P_{a_i,0}.
\end{eqnarray*}
Since $\eta\in \Omega(G-D)$ and $\eta$ has $q^\ell(q-3)$ poles, it follows that $d_\Omega = q^\ell(q-3)$.
For $\ell=1$, the values $q=5$ and $q=7$ produce the codes with parameters $[124, 107 ,10]$ over $\mathbb{F}_{25}$ and $[342, 296, 28]$ over $\mathbb{F}_{49}$, respectively.
The first code attains the best known minimum distance and the second one improves the best known minimum distance.
\end{example}

\begin{example}
Let $q$ be a prime power, $n$ be an even integer, and $m$ be a divisor of $q^n-1$ such that $\gcd(m, q^{n/2}-1) = 1$.
In~\cite[Example 11]{GQ} Garcia and Quoos considered the curve defined by the affine equation
$
y^m=(x^{q^{n/2}}-x)^{q^{n/2}-1}
$
over $\mathbb{F}_{q^n}$, and showed that the genus and the number of $\mathbb{F}_{q^n}$-rational points of this curve are
$g=(q^{n/2}-1)(m-1)/2$ and $N=(q^n-q^{n/2})m+(q^{n/2}+1)$,
respectively.
When $n=2$, $m=3$ and $q=5$, we have the curve $y^3=x^5-x$ with $g=4$ and $N=66$ over $\mathbb{F}_{25}$. By Theorem~\ref{pesos}, we have that $H(P_\infty)=\langle 3,5 \rangle$, and so the one-point codes $C_{\mathcal{L}}$ with $G=5P_\infty$ and $G=6P_\infty$ have parameters $[65,3,60]$ and $[65,4,59]$, respectively. Both these examples improve the best known minimum distance. We used  Magma~\cite{Magma} in this example to calculate the parameters of the codes.

\end{example}

 \begin{example}\label{ex1maximal}
The curve defined by $y^{q+1}=\sum\limits_{i=1}^t x^{q/2^i}$ with $q=2^t$ over $\mathbb{F}_{q^2}$ and genus $g=q(q-2)/4$ is a maximal curve~\cite{AT}. In this case, we have that $m=q+1$ and $r=q/2$. For a totally ramified place $P\neq P_\infty$, an application of Theorem~\ref{casobom} gives that $H(P)=\langle q-1,q,q+1\rangle$ and $H(P_\infty)=\langle q+1,q/2\rangle$.
When $t=3$, we have the  maximal curve $y^9=x^4+x^2+x$ over $\mathbb{F}_{64}$ with $257$ rational points. As $H(P_\infty)=\langle 4,9\rangle$ and $H(P)=\langle7,8,9\rangle$, we have that
$G(P_\infty)=\{1,2,3,5,6,7,10,11,14,15,19,23\}$,
$G(P)=\{1,2,3,4,5,6,10,11,12,13,19,20\}$ and
$\Gamma(P_\infty,P)=\{(1,20),(2,13),(3,6),(5,19),(6,12),$ $ (7,5),(10,11),(11,4),(14,10),(15,3),(19,2),(23,1)\}.$
From Theorem~\ref{puregaps}, we can check that $(10, 10)$ is a pure gap. By using Theorem~\ref{puregapscodes} and taking $t_1=t_2=0$ and $(\beta,\gamma)=(10,10)$, we obtain $G=19P_\infty+19P$. In this way, we construct two-point codes  $C_{\Omega}$ with parameters $[255,228,\geq 18]$. The minimum distance of this code exceeds the minimum distance of the best known linear code over $\mathbb{F}_{64}$ with the same dimension and length by two units. We recall that, for an $[n,k,d]$ linear code and each non-negative integer $s<k$, there exists a  linear code with parameters $[n-s,k-s,d]$. So, for $[255-s,228-s,\geq 18]$ with $s\in \{0,1,2,\ldots,29\}$, we obtain $30$ new linear codes with the best known parameters.
\end{example}

\section*{Acknowledgements}
The authors would like to thank the anonymous referees for their useful comments and important corrections on an earlier version of this work, including generalizing a previous statement of Theorem 4.3.

 \end{document}